\documentclass[reqno]{amsart}
\usepackage{graphicx}

\usepackage{epstopdf}
\usepackage[justification=centering]{caption}
\usepackage{lineno,hyperref}
\usepackage{amsfonts,amsmath,latexsym,amssymb,amsthm}
\usepackage{geometry}
\newcommand{\rs}{{\mathcal R\,}}
\newcommand{\hs}{{\mathcal H\,}}

\modulolinenumbers[5]
\usepackage{makecell,multirow,diagbox}
\hypersetup{
	colorlinks=true,
	linkcolor=blue
}
\hypersetup{
	citecolor=blue
}

\newtheorem{theorem}{Theorem}[section]
\newtheorem{lemma}{Lemma}[section]
\usepackage[justification=centering]{caption}
\newtheorem{corollary}[theorem]{Corollary}

\theoremstyle{remark}
\newtheorem{remark}[theorem]{Remark}

\numberwithin{equation}{section}
\allowdisplaybreaks[4]
\begin{document}

\title{Splitting of operator for frame inequalities in Hilbert spaces}

%    Information for first author
\author{Dongwei Li}
%    Address of record for the research reported here
\address{School of Mathematics, HeFei University of Technology, 230009, P. R. China}
%    Current address
\email{dongweili@huft.edu.cn}
%    \thanks will become a 1st page footnote.
%\thanks{All authors were supported  by the National Natural %Science Foundation of China
% Grant \#11271001 and \#61370147.}

%    Information for second author
%\author{Jinsong Leng}
%\address{School of Mathematical Sciences, University of Electronic Science and Technology of China, 611731, P. R. China}
%\email{jinsongleng@126.com}
%\thanks{Support information for the second author.}
%    General info

%\date{January 1, 2001 and, in revised form, June 22, 2001.}

%\dedicatory{This paper is dedicated to our advisors.}


\begin{abstract}
	In this paper, we obtain  a new type of inequalities  for frames, which are parametrized by a parameter $\lambda\in\rs$. By suitable choices of  $\lambda$, one obtains the previous results as special cases. Our new proof also makes the underlying mathematical structure that gives rise to these inequalities more transparent than previous approaches: Our proof shows that the main point is the splitting $S = S_1+ S_2$ of the positive deﬁnite frame operator S into the two positive semideﬁnite operators $S_1$ and $S_2$.
\end{abstract}

\maketitle

\section{Introduction}
Frames in Hilbert spaces were first introduced in 1952 by Duffin and Schaeffer \cite{duffin1952class} to study some deep problems in nonharmonic Fourier series,  reintroduced in 1986 by Daubechies, Grossmann and Meyer \cite{daubechies1986painless}, and today frames play important roles in many applications in several areas of mathematics, physics, and engineering, such as coding theory \cite{leng2011optimal,li2018frame}, sampling theory \cite{zhao2006perturbation}, quantum measurements \cite{eldar2002optimal},filter bank theory \cite{kovacevic2002filter} and image processing \cite{donoho2005continuous}.

Let $\hs$ be a separable space and $I$ a countable index set. A sequence  $\{f_i\}_{i\in I}$ of elements of $\hs$ is a frame for $\hs$ if there exist constants $A,~B>0$ such that
$$A\|f\|^2\le \sum_{i\in I}\left| \left\langle f,f_i\right\rangle \right| ^2\le B\|f\|^2,~~~\forall f\in\hs.$$
The number $A,~B$ are called lower and upper frame bounds, respectively. If $A=B$, then this frame is called an $A$-tight frame, and if $A=B=1$, then it is called a Parseval frame.

Suppose  $\{f_i\}_{i\in I}$ is a frame for $\hs$, then the frame operator is a self-adjoint positive invertible operators, which is given by
$$S:\hs\rightarrow\hs,~~Sf=\sum_{i\in I}\left\langle f,f_i\right\rangle f_i.$$ 
The following reconstruction formula holds:
$$f=\sum_{i\in I}\left\langle f,f_i\right\rangle S^{-1}f_i=\sum_{i\in I}\left\langle f,S^{-1}f_i\right\rangle f_i,$$
where the family $\{\widetilde{f}_i\}_{i\in I}=\{S^{-1}f_i\}_{i\in I}$ is also a frame for $\hs$, which is called the canonical dual frame of $\{f_i\}_{i\in I}$.
The frame $\{g_i\}_{i\in I}$ for $\hs$ is called an alternate dual frame of $\{f_i\}_{i\in I}$ if the following formula holds:
$$f=\sum_{i\in I}\left\langle f,f_i\right\rangle g_i=\sum_{i\in I}\left\langle f,g_i\right\rangle f_i$$
for all $f\in\hs$ \cite{han2000frames}.

Let $\{f_i\}_{i\in I}$ be a frame for $\hs$, for every $J\subset I$, we define the operator
$$S_J=\sum_{i\in J}\left\langle f,f_i\right\rangle f_i,$$
and denote $J^c=I\setminus J$.
 
 In \cite{balan2006signal}, the authors solved a long-standing conjecture of the signal processing community.  They showed that for suitable frames $\{f_i\}_{i\in I}$, a signal $f$ can (up to a global phase) be recovered from the phase-less measurements $\{|\langle f,f_i\rangle|\}_{i\in I}$. Note, that this only shows that reconstruction of $f$ is in principle possible, but there is not an effective constructive algorithm. While searching for such an algorithm, the authors of \cite{balan2005decompositions} discovered a new identity for Parseval frames \cite{balan2007new}. The authors in \cite{guavructa2006some,zhu2010note} generalized these identities to alternate dual frames and got some general results.
 The study of inequalities has interested many mathematicians.  Some authors have extended the equalities and inequalities for frames in Hilbert spaces to generalized  frames \cite{li2018some,li2012some,poria2017some}. The following form was given in \cite{balan2007new} (See \cite{balan2005decompositions} for a discussion of the origins of this fundamental identity).
 
\begin{theorem}\label{thm1.1}
Let $\{f_i\}_{i\in I}$ be a Parseval frame for $\hs$. For every $J\subset I$ and every $f\in\hs$, we have
\begin{equation}\label{1.2}
\sum_{i\in J}\left| \left\langle f,f_i\right\rangle \right| ^2+\left\| \sum_{i\in J^c}\left\langle f,f_i\right\rangle f_i\right\| ^2=\sum_{i\in J^c}\left| \left\langle f,f_i\right\rangle \right| ^2+\left\| \sum_{i\in J}\left\langle f,f_i\right\rangle f_i\right\| ^2\ge \frac{3}{4}\|f\|^2.
\end{equation}
\end{theorem}
 
 Later on, the author in \cite{guavructa2006some} generalized Theorem \ref{thm1.1} to general frames.
 
\begin{theorem}\label{thm1.2}
	Let $\{f_i\}_{i\in I}$ be a frame for $\hs$ with canonical dual frame $\{\widetilde{f}_i\}_{i\in I}$. Then for every $J\subset I$ and every $f\in\hs$, we have
	\begin{equation}\label{1.3}
	\sum_{i\in J}\left| \left\langle f,f_i\right\rangle \right| ^2+\sum_{i\in I}\left| \left\langle S_{J^c}f,\widetilde{f}_i\right\rangle \right| ^2=\sum_{i\in J^c}\left| \left\langle f,f_i\right\rangle \right| ^2+\sum_{i\in I}\left| \left\langle S_{J}f,\widetilde{f}_i\right\rangle \right| ^2\ge \frac{3}{4}\sum_{i\in I}\left| \left\langle f,f_i\right\rangle \right| ^2.
	\end{equation}
\end{theorem}
\begin{theorem}\label{thm1.3}
	Let $\{f_i\}_{i\in I}$ be a frame for $\hs$ and $\{g_i\}_{i\in I}$ be an alternate dual frame of $\{f_i\}_{i\in I}$. Then for every $J\subset I$ and every $f\in\hs$, we have
	\begin{equation}\label{1.4}
{\rm Re}\left( \sum_{i\in J}\left\langle f,g_i\right\rangle \overline{\left\langle f,f_i\right\rangle }\right) +\left\| \sum_{i\in J^c}\left\langle f,g_i\right\rangle f_i\right\| ^2={\rm Re}\left( \sum_{i\in J^c}\left\langle f,g_i\right\rangle \overline{\left\langle f,f_i\right\rangle }\right) +\left\| \sum_{i\in J}\left\langle f,g_i\right\rangle f_i\right\| ^2\ge \frac{3}{4}\|f\|^2.
	\end{equation}
\end{theorem}
%Motivated by these interesting results, the authors in \cite{zhu2010note} generalized the Theorem \ref{thm1.3} to a more general form which does not involve the real parts of the complex numbers.
%
%\begin{theorem}\label{thm1.4}
%	Let $\{f_i\}_{i\in I}$ be a  frame for $\hs$ and $\{g_i\}_{i\in I}$ be an alternate dual frame of $\{f_i\}_{i\in I}$. Then for every $J\subset I$ and every $f\in\hs$, we have
%	\begin{equation*}
%	\left( \sum_{i\in J}\left\langle f,g_i\right\rangle \overline{\left\langle f,f_i\right\rangle }\right) +\left\| \sum_{i\in J^c}\left\langle f,g_i\right\rangle f_i\right\| ^2=\overline{\left( \sum_{i\in J^c}\left\langle f,g_i\right\rangle \overline{\left\langle f,f_i\right\rangle }\right)} +\left\| \sum_{i\in J}\left\langle f,g_i\right\rangle f_i\right\| ^2\ge \frac{3}{4}\|f\|^2.
%	\end{equation*}
%\end{theorem}

In this paper, we first consider the splitting $S=S_J+S_{J^c}$ of  frame operator  and study the properties of splittings. Then we generalized the above inequalities to a more general form which involve a scalar $\lambda\in\rs$. These inequalities involve the expressions $\langle S_Jf,f \rangle$, $\|S_Jf\|$, etc., where $S_J$ is a ``truncated form'' of the  frame operator. 
\section{Results and New proofs}
\begin{theorem}\label{lem1}
	Let $S:\hs\rightarrow\hs$ be a bound, self-adjoint positive definite operator. Furthermore, let $S_1, S_2:\hs\rightarrow\hs$ be bounded, self-adjoint, and positive semidefinite with $S=S_1+S_2$. Then the following are true:
	\begin{enumerate}
		\item  For $i\in\{1,2\}$, we have  $0\le S_iS^{-1}S_i$.
		\item We have $S_2+S_1S^{-1}S_1\le S$.
		\item We have $S_1S^{-1}S_1+S_2S^{-1}S_2\le S$.
		\item $S_2+S_1S^{-1}S_1=S_1+S_2S^{-1}S_2$.
		\item If $p,~q\in\rs$ are chosen such that $\varrho(a):=a^2-a\cdot(q-p-1)+1-q\ge 0$ for all $a\in[0,1]$, then we have
		$$p\cdot S_1+q\cdot S_2\le S_2+S_1S^{-1}S_1.$$
		\item If $p,~q\in\rs$ are chosen such that $\eta(a):=a^2-a(1+p)+q+p\ge 0 $ for all $a\in[0,1]$, then we have
		$$S_1-S_1S^{-1}S_1\le p\cdot S_2+q\cdot S_2.$$
		\item If $p,~q\in\rs$ are chosen such that  $\tau(a):=a^2+a\cdot (\frac{q-p}{2}-1)+\frac{1-q}{2}\ge 0$ for all $a\in[0,1]$, then we have
		$$p\cdot S_1+q\cdot S_2\le S_1S^{-1}S_1+S_2S^{-1}S_2.$$
	\end{enumerate}
	In all of these statement, we write $U\le V$ for all operators $U,~V:\hs\rightarrow\hs$ if $U,~V$ are self-adjoint, and if furthermore $V-U$ is positive semidefinite.
\end{theorem}
\begin{proof}
We first prove the following elementary fact: If $P:\hs\rightarrow\hs$ is a bounded positive definite operator, then a self-adjoint, bounded operator $X:\hs\rightarrow\hs$ is positive semidefinite if and only if $PXP$ is. Indeed, if $X$ is positive semidefinite, then $\langle PXPf,f\rangle=\langle XPf,Pf\rangle$ for all $f\in\hs$, so that $PXP$ is positive semidefinite. Conversely, if $PXP\ge 0$, we can apply what we just showed with $P^{-1}$ instead of $P$ to see $X=P^{-1}(PXP)P^{-1}\ge 0$. Overall, this means
\begin{equation}\label{2.1}
\forall U,~V:\hs\rightarrow\hs {\rm ~self-adjoint~ and~ }~P:\hs \rightarrow\hs~{\rm ~positive~definite~ }:~U\le V\Leftrightarrow PUP\le PVP.
\end{equation}
Note that $S^{-1/2}$ is positive definite and bounded, so that the operators $U:=S^{-1/2}S_1S^{-1/2}$ and $U:=S^{-1/2}S_12S^{-1/2}$ are positive semidefinite and bounded. Furthermore,
\begin{equation}\label{2.2}
U+V=S^{-1/2}(S_1+S_2)S^{-1/2}=S^{-1/2}SS^{-1/2}=I_{\hs}.
\end{equation}
Now, we properly start the proof:

(1). Since $U,~V$ are positive semidefinite, we have $0\le U\le U+V=I_{\hs}$, and thus $I_{\hs}-U\ge 0$. Since $I_{\hs}-U$ and $U$ commute, this implies $U-U^2=U\cdot (I_{\hs}-U)\ge 0$, i.e., $0\le U^2\le U$. In view of \eqref{2.1}, this implies $0\le S^{1/2}U^2S^{1/2}\le S^{1/2}US^{1/2}$. But since $S^{1/2}US^{1/2}=S_1$ and $S^{1/2}U^2S^{1/2}=S_1S^{-1}S_1$, this implies the claim of the first part for $i=1$. The proof for $i=2$ is similar.

(2). In view of \eqref{2.2} (with $P=S^{-1/2}$), in view of the definition of $U,~V$ and because of $V=I_{\hs}-U$ (see \eqref{2.2}), we have the following equivalence:
$$S_2+S_1S^{-1}S_1\le S\Leftrightarrow V+UU\le I_{\hs}.$$
But we saw in the previous part that $U^2\le U$, so that $V+UU\le V+U=I_{\hs}$ does hold.

(3). Part (1) shows $S_iS^{-1}S_i\le S_i$ for $i\in{1,2}$. Hence, $S_1S^{-1}S_1+S-2S^{-1}S_2\le S_1+S_2=S$.

(4). By multiplying from the left and from the right by $S^{-1/2}$, we see that the claimed identity is equivalent to $V+UU=U+VV$. Because of $V=I_{\hs}
-V$, this is in tun equivalent to 
$$I_{\hs}-U+UU=U+(I_{\hs}-U)(I_{\hs}-U),$$
which is easy seen to be true by expanding the right-hand side.

(5). In view of \eqref{2.1} (with $P=S^{-1/2}$), from the definition of $U,~V$, and because of $V=I_{\hs}-U$ (see \eqref{2.2}), we have the following equivalence:
\begin{align*}
p\cdot S_1+q\cdot S_2\le S_2+S_1S^{-1}S_1&\Leftrightarrow p\cdot U+q\cdot V\le V+UU\\
&\Leftrightarrow p\cdot U+q\cdot I_{\hs}-q\cdot U\le I_{\hs}-U+UU\\
&\Leftrightarrow U^2+U\cdot (q-p-1)+(1-q)\cdot I_{\hs}\ge 0\\
&\Leftrightarrow\varrho(U)\ge 0.
\end{align*}
But in the proof of part (1) of the Theorem \ref{lem1}, we saw $0\le U\le I_{\hs}$. Since we have $\varrho\ge 0$ on $[0,1]$ by assumption, elementary properties of the spectral calculus (see e.g. \cite[Theorem 4.2]{lang1993real}) imply that $\varrho(U)$ is positive semidefinite, as desired.

(6). Just as in the proof of the previous part, we get the following equivalence:
\begin{align*}
S_1-S_1S^{-1}S_1\le p\cdot S_2+q\cdot S&\Leftrightarrow U-UU\le p\cdot V+q\cdot I_{\hs}\\
&\Leftrightarrow U-U^2\le (p+q)\cdot I_{\hs}-p\cdot U\\
&\Leftrightarrow U^2-(1+p)\cdot U+(p+q)\cdot I_{\hs}\ge 0\\
&\Leftrightarrow\eta(U)\ge 0.
\end{align*}
Again, just as in the proof of the previous part, we see that $\eta(U)$ is indeed positive semidefinite, since $0\le U\le I_{\hs}$ and since $\eta\ge 0$ on $[0,1]$ by assumption.

(7). Just as in part (5) of the Theorem \ref{lem1}, we get the following equivalence:
\begin{align*}
p\cdot S_1+q\cdot S_2\le S_1S^{-1}S_1+S_2S^{-1}S_2&\Leftrightarrow p\cdot U+q\cdot V\le UU+VV\\
&\Leftrightarrow p\cdot U+q\cdot (I_{\hs}-U)\le UU+(I_{\hs}-U)(I_{\hs}-U)\\
&\Leftrightarrow 2U^2+U\cdot (q-p-2)+I_{\hs}\cdot (1-q)\ge 0\\
&\Leftrightarrow\tau(U)\ge 0.
\end{align*}
Again, just as in the proof of the previous part, we see that $\tau(U)$ is indeed positive semidefinite, since $0\le U\le I_{\hs}$ and since $\tau\ge 0$ on $[0,1]$ by assumption.
\end{proof}

\begin{theorem}\label{them2.1}
Let $S:\hs\rightarrow\hs$ be a bound, self-adjoint positive definite operator. Furthermore, let $S_1, S_2:\hs\rightarrow\hs$ be bounded, self-adjoint, and positive semidefinite with $S=S_1+S_2$.  Then for any $\lambda\in\rs$, we have
\begin{equation}\label{2.3}
\left( \lambda-\frac{\lambda^2}{4}\right)\cdot S_1+\left( 1-\frac{\lambda^2}{4}\right) \cdot S_2\le S_2+S_1S^{-1}S_1=S_1+S_2S^{-1}S_2\le S.
\end{equation}
\end{theorem}
\begin{proof}
The middle identity is a direct consequence of part (4) of Theorem \ref{lem1}. Likewise, the final estimate follows directly from part (2) of Theorem \ref{lem1}.

To prove the first part of the equation \eqref{2.3}, we want to apply part (5) of Theorem \ref{lem1} with the choices $p=\lambda-\frac{\lambda^2}{4}$ and $q=1-\frac{\lambda^2}{4}$. With these choices, the polynomial $\varrho$ from Theorem \ref{lem1} takes the form
$$\varrho(a)=a^2+a\cdot(q-p-1)+1-q=a^2-\lambda a+\frac{\lambda^2}{4}=\left( a-\frac{\lambda}{2}\right) ^2,$$
so that $\varrho(a)\ge 0$ for all $a\in[0,1]$, as required by part (5) of Theorem \ref{lem1}. An application of that part of the Theorem \ref{lem1} completes the proof.
\end{proof}
By choosing $S$ to be the frame operator, and by choosing $S_1:=S_J$ and $S_2:=S_{J^c}$, we see that $S$, $S_1$ and $S_2$ are all bounded, self-adjoint, positive semi-definite, that $S$ is positive definite, and that $S=S_1+S_2$. Furthermore, directly from the definitions, we see
 \begin{align}\label{2.4}
\langle S_if,f\rangle&=\sum_{i\in J_i}\left\| \langle f,f_i\rangle\right\|^2,\nonumber\\
 \langle Sf,f\rangle&=\sum_{i\in I}\left\| \langle f,f_i\rangle\right\|^2,\\
\langle S_iS^{-1}S_if,f_i\rangle=\langle S(S^{-1}S_if),S^{-1}S_if\rangle &=\sum_{i\in I}\left\| \langle S^{-1}S_if,f_i\rangle\right\|^2.\nonumber
 \end{align}
\begin{corollary}\label{cor6}
Let $\{f_i\}_{i\in I}$ be a frame for $\hs$ with frame operator $S$. Then for any $\lambda\in\rs$, for all $J\subset I$, and any $f\in\hs$, we have
 \begin{align}\label{2.5}
\left( \lambda-\frac{\lambda^2}{4}\right)\cdot \sum_{i\in J}\left\| \langle f,f_i\rangle\right\|^2+	\left( 1-\frac{\lambda^2}{4}\right) \cdot\sum_{i\in J^c}\left\| \langle f,f_i\rangle\right\|^2&\le \sum_{i\in J^c}\left\| \langle f,f_i\rangle\right\|^2+\sum_{i\in I}\left\| \langle S^{-1}S_Jf,f_i\rangle\right\|^2\nonumber\\
&=\sum_{i\in J}\left\| \langle f,f_i\rangle\right\|^2+\sum_{i\in I}\left\| \langle S^{-1}S_{J^c}f,f_i\rangle\right\|^2\nonumber\\
&\le \sum_{i\in I}\left\| \langle f,f_i\rangle\right\|^2.
\end{align}
\end{corollary}
\begin{proof}
We choose $S_1,~S_2$ as outlined before equation \eqref{2.4}. In view of the ``translation table'' in equation \eqref{2.4}, and by the definition of the relation ``$U\le V$'' for self-adjoint operator $U,~V$, the equation \eqref{2.5} is equivalent to 
 \eqref{2.3}.  By Theorem \ref{them2.1}, the result holds.
\end{proof}
\begin{remark}
If we take $\lambda=1$ in \eqref{2.5}, Corollary \ref{cor6} is equal to Theorem \ref{thm1.2}. If we consider $S$ as a fusion frame operator in Theorem \ref{them2.1}, we can easy get the  \cite[Theorem 3]{li2017somee}. If we consider $S$ as a HS-frame operator in Theorem \ref{them2.1}, we can easy get the \cite[Theorem 3.5]{poria2017some}.
\end{remark}

\begin{theorem}\label{them7}
	Let $S:\hs\rightarrow\hs$ be a bound, self-adjoint positive definite operator. Furthermore, let $S_1, S_2:\hs\rightarrow\hs$ be bounded, self-adjoint, and positive semidefinite with $S=S_1+S_2$.  Then for any $\lambda\in\rs$, we have
	\begin{equation}\label{2.6}
0\le S_1-S_1S^{-1}S_1\le (\lambda-1)\cdot S_2+\left( 1-\frac{\lambda}{2}\right) ^2\cdot S.
	\end{equation}
\end{theorem}
\begin{proof}
The first estimate of equation \eqref{2.6} is a direct consequence of part (1) of Theorem \ref{lem1}. 
To prove the second estimate, we want to apply part (6) of Theorem \ref{lem1}, with $p=\lambda-1$ and $q=\left( 1-\frac{\lambda}{2}\right) ^2=1-\lambda+\frac{\lambda^2}{4}$. With these choices, the polynomial $\eta$ from the Theorem \ref{lem1} takes the form 
$$\eta(a)=a^2-a\cdot(1+p)+q+p=a^2-\lambda a+\frac{\lambda^2}{4}=\left( a-\frac{\lambda}{2}\right) ^2,$$
so that $\eta(a)\ge 0$ for all $a\in[0,1]$, as required by part (6) of  Theorem \ref{lem1}. An application of that theorem thus finishes the proof.
\end{proof}
\begin{corollary}Let $\{f_i\}_{i\in I}$ be a frame for $\hs$ with frame operator $S$. Then for any $\lambda\in\rs$, for all $J\subset I$, and any $f\in\hs$, we have
	\begin{align*}
0&\le \sum_{i\in J}\left\| \langle f,f_i\rangle\right\|^2-\sum_{i\in I}\left\| \langle S^{-1}S_Jf,f_i\rangle\right\|^2
	\le (\lambda-1)\cdot\sum_{i\in J^c}\left\| \langle f,f_i\rangle\right\|^2+\left( 1-\frac{\lambda}{2}\right) ^2\cdot\sum_{i\in I}\left\| \langle f,f_i\rangle\right\|^2.
	\end{align*}
\end{corollary}
\begin{proof}
By choosing $S_1=S_J$ and $S_2=S_{J^c}$, and by using the ``translation table'' given in equation \eqref{2.4}, we see that the claim is equivalent to \eqref{2.6}, and result holds by Theorem \ref{them7}.
\end{proof}
\begin{remark}
If we consider $S$ as a fusion frame operator in Theorem \ref{them7}, we can easy get the  \cite[Theorem 5]{li2017somee}. If we consider $S$ as a g-frame operator in Theorem \ref{them7} for Hilbert C*-modules, we can easy get the  \cite[Theorem 2.4]{xiang2016new}. 
\end{remark}
\begin{theorem}\label{thm10}
	Let $S:\hs\rightarrow\hs$ be a bound, self-adjoint positive definite operator. Furthermore, let $S_1, S_2:\hs\rightarrow\hs$ be bounded, self-adjoint, and positive semidefinite with $S=S_1+S_2$.  Then for any $\lambda\in\rs$, we have
\begin{equation}\label{2.7}
\left( 2\lambda-\frac{\lambda^2}{2}-1\right)\cdot S_1+\left( 1-\frac{\lambda^2}{2}\right)\cdot S_2\le S_1S^{-1}S_1+S_2S^{-1}S_2\le S.
\end{equation}
\end{theorem}
\begin{proof}
The second of these inequalities is a direct consequence of part (3) of Theorem \ref{lem1}.
To prove the first estimate, we want to involve part (7) of Theorem \ref{lem1} with $p=2\lambda-\frac{\lambda^2}{2}-1$ and $q=1-\frac{\lambda^2}{2}$. With these choices, the polynomial $\tau$ from Theorem \ref{lem1} takes the form
$$\tau(a)=a^2+a\cdot \left( \frac{q-p}{2}-1\right) +\frac{1-q}{2}=a^2-\lambda a+\frac{\lambda^2}{4}=\left( a-\frac{\lambda}{2}\right) ^2,$$
so that $\tau(a)\ge 0$ for all $a\in[0,1]$, as required in part (7)  of Theorem \ref{lem1}. An application of that theorem thus finishes the proof.
\end{proof}

\begin{corollary}Let $\{f_i\}_{i\in I}$ be a frame for $\hs$ with frame operator $S$. Then for any $\lambda\in\rs$, for all $J\subset I$, and any $f\in\hs$, we have
	\begin{align*}
	\left( 2\lambda-\frac{\lambda^2}{2}-1\right)\cdot  \sum_{i\in J}\left\| \langle f,f_i\rangle\right\|^2+\left( 1-\frac{\lambda^2}{2}\right)\cdot \sum_{i\in J^c}\left\| \langle f,f_i\rangle\right\|^2
	&\le \sum_{i\in I}\left\| \langle S^{-1}S_Jf,f_i\rangle\right\|^2+\sum_{i\in I}\left\| \langle S^{-1}S_{J^c}f,f_i\rangle\right\|^2\nonumber\\
&	\le \sum_{i\in I}\left\| \langle f,f_i\rangle\right\|^2.
	\end{align*}
\end{corollary}
\begin{proof}
By choosing $S_1=S_J$ and $S_2=S_{J^c}$, and by using the ``translation table'' given in equation \eqref{2.4}, we see that the claim is equivalent to \eqref{2.7}. Then the result holds by Theorem \ref{thm10}.
\end{proof}
\begin{remark}
If we consider $S$ as a continue fusion frame operator in Theorem \ref{them7}, we can easy get the  \cite[Theorem 2.13]{li2018some}. If we consider $S$ as a g-frame operator in Theorem \ref{them7} for Hilbert C*-modules, we can easy get the  \cite[Theorem 2.4]{xiang2016new}. 
\end{remark}

Next, we give a new type of inequality of frames of Theorem \ref{thm1.3}. We first need follow lemma.
\begin{lemma}\label{lem2}
Let $U,V$ be two bounded linear operator in $\hs$  and $U+V=I_{\hs}$, then for any $\lambda\in\rs$, we have
$$U^*U+\lambda\cdot (V^*+V)\ge\lambda(2-\lambda)\cdot I_{\hs}.$$
\end{lemma}
\begin{proof}
Since $U+V=I_{\hs}$, we have 
	\begin{align*}
U^*U+\lambda(V^*+V)&=U^*U-\lambda(U^*+U)+2\lambda\cdot I_{\hs}\\
&=U^*U-\lambda\cdot (U^*+U)+2\lambda\cdot I_{\hs}+\lambda^2\cdot I_{\hs}-\lambda^2\cdot I_{\hs}\\
&=(U-\lambda\cdot I_{\hs})^*(U-\lambda\cdot I_{\hs})+\lambda(2-\lambda)\cdot I_{\hs}\\
&	\ge\lambda(2-\lambda)\cdot I_{\hs}.
\end{align*}
\end{proof} 
\begin{theorem}\label{thm14}
Let $\{f_i\}_{i\in I}$ be a frame for $\hs$ and $\{g_i\}_{i\in I}$ be an alternate dual frame of $\{f_i\}_{i\in I}$. Then for any $\lambda\in\rs$, for all $J\subset I$, and any $f\in\hs$, we have
	\begin{equation*}
{\rm Re}\left( \sum_{i\in J}\left\langle f,g_i\right\rangle \overline{\left\langle f,f_i\right\rangle }\right) +\left\| \sum_{i\in J^c}\left\langle f,g_i\right\rangle f_i\right\| ^2\ge (2\lambda-\lambda^2)\cdot {\rm Re}\left( \sum_{i\in J}\left\langle f,g_i\right\rangle \overline{\left\langle f,f_i\right\rangle }\right) +(1-\lambda^2)\cdot {\rm Re}\left( \sum_{i\in J^c}\left\langle f,g_i\right\rangle \overline{\left\langle f,f_i\right\rangle }\right) 
\end{equation*}
\end{theorem}
\begin{proof}
For any $J\subset I$ and $f\in\hs$, we define operators $U,V$ as
$$Uf=\sum_{i\in J^c}\left\langle f,g_i\right\rangle f_i,~~~Vf=\sum_{i\in J}\left\langle f,g_i\right\rangle f_i.$$
Clearly, $U,~V$ are bounded linear operator and $U+V=I_{\hs}$. From Lemma \ref{lem2}, for any $f\in\hs$, we have 
$$\left\langle U^*Uf,f\right\rangle +\lambda\overline{\left\langle Vf,f\right\rangle}+\lambda\left\langle Vf,f\right\rangle\ge (2\lambda-\lambda^2)\left\langle I_{\hs}f,f\right\rangle,\eqno(2.13)$$
and then,
$$\|Uf\|^2+2\lambda{\rm Re}\left\langle Vf,f\right\rangle \ge (2\lambda-\lambda^2){\rm Re}\left\langle I_{\hs}f,f\right\rangle,$$
which implies
\begin{align*}
\|Uf\|^2&\ge (2\lambda-\lambda^2){\rm Re}\left\langle I_{\hs}f,f\right\rangle-2\lambda{\rm Re}\left\langle Vf,f\right\rangle\\
&=(2\lambda-\lambda^2){\rm Re}\left\langle (U+V)f,f\right\rangle-2\lambda{\rm Re}\left\langle Vf,f\right\rangle\\
&=(2\lambda-\lambda^2){\rm Re}\left\langle Uf,f\right\rangle -\lambda^2{\rm Re}\left\langle Vf,f\right\rangle \\
&=(2\lambda-\lambda^2){\rm Re}\left\langle Uf,f\right\rangle+(1-\lambda^2){\rm Re}\left\langle Vf,f\right\rangle -{\rm Re}\left\langle Vf,f\right\rangle.
\end{align*}
Hence 
$$\|Uf\|^2+{\rm Re}\left\langle Vf,f\right\rangle\ge (2\lambda-\lambda^2){\rm Re}\left\langle Uf,f\right\rangle+(1-\lambda^2){\rm Re}\left\langle Vf,f\right\rangle,$$
thus
	\begin{equation*}
{\rm Re}\left( \sum_{i\in J}\left\langle f,g_i\right\rangle \overline{\left\langle f,f_i\right\rangle }\right) +\left\| \sum_{i\in J^c}\left\langle f,g_i\right\rangle f_i\right\| ^2\ge (2\lambda-\lambda^2)\cdot {\rm Re}\left( \sum_{i\in J}\left\langle f,g_i\right\rangle \overline{\left\langle f,f_i\right\rangle }\right) +(1-\lambda^2)\cdot {\rm Re}\left( \sum_{i\in J^c}\left\langle f,g_i\right\rangle \overline{\left\langle f,f_i\right\rangle }\right) 
\end{equation*}
\end{proof}
In the sequel we give a more general result. Consider a bounded sequence of complex numbers $\{a_i\}_{i\in I}$. In Theorem \ref{thm14} we take 
$$Uf=\sum_{i\in J^c}a_i\left\langle f,g_i\right\rangle f_i,~~~Vf=\sum_{i\in J}(1-a_i)\left\langle f,g_i\right\rangle f_i.$$
We can get the following result.
\begin{corollary}
Let $\{f_i\}_{i\in I}$ be a frame for $\hs$ and $\{g_i\}_{i\in I}$ be an alternate dual frame of $\{f_i\}_{i\in I}$. Then for all bounded sequence $\{a_i\}_{i\in I}$ we have
\begin{align*}
&{\rm Re}\left( \sum_{i\in J}(1-a_i)\left\langle f,g_i\right\rangle \overline{\left\langle f,f_i\right\rangle }\right) +\left\| \sum_{i\in J^c}a_i\left\langle f,g_i\right\rangle f_i\right\| ^2\\
&\ge (2\lambda-\lambda^2)\cdot {\rm Re}\left( \sum_{i\in J}(1-a_i)\left\langle f,g_i\right\rangle \overline{\left\langle f,f_i\right\rangle }\right) +(1-\lambda^2)\cdot {\rm Re}\left( \sum_{i\in J^c}a_i\left\langle f,g_i\right\rangle \overline{\left\langle f,f_i\right\rangle }\right) 
\end{align*}
\end{corollary}
 \begin{remark}
 If we take $\lambda=\frac{1}{2}$ Theorem \ref{thm14}, we can obtain the inequality in Theorem \ref{thm1.3}.
 \end{remark}

%\begin{lemma}\label{lem3}
%	Let $U,V$ be two bounded linear operator in $\hs$  and $U+V=I_{\hs}$, then for any $\lambda\in\rs$, we have
%	$$U^*U+\lambda\cdot (V^*+V)\ge\lambda(2-\lambda)\cdot I_{\hs}.$$
%\end{lemma}

%\begin{theorem}\label{thm17}
%	Let $\{f_i\}_{i\in I}$ be a  frame for $\hs$ and $\{g_i\}_{i\in I}$ be an alternate dual frame of $\{f_i\}_{i\in I}$. Then for every $J\subset I$ and every $f\in\hs$, we have
%	\begin{equation*}
%	\left( \sum_{i\in J}\left\langle f,g_i\right\rangle \overline{\left\langle f,f_i\right\rangle }\right) +\left\| \sum_{i\in J^c}\left\langle f,g_i\right\rangle f_i\right\| ^2=\overline{\left( \sum_{i\in J^c}\left\langle f,g_i\right\rangle \overline{\left\langle f,f_i\right\rangle }\right)} +\left\| \sum_{i\in J}\left\langle f,g_i\right\rangle f_i\right\| ^2\ge \frac{3}{4}\|f\|^2.
%	\end{equation*}
%\end{theorem}
%\section*{Acknowledgements}
%The research is supported by the National Natural Science Foundation of China (LJT10110010115). 
\bibliographystyle{plain}
\bibliography{my}

\end{document}